\documentclass[11pt]{article}

\usepackage[margin=1in]{geometry}
\usepackage{amsmath,amssymb,amsthm}
\usepackage[T1]{fontenc}
\usepackage[colorlinks=true,linkcolor=black,citecolor=black,urlcolor=blue!55!black]{hyperref}

\newcommand{\T}{\mathbb{T}}
\newcommand{\abs}[1]{\left\lvert #1\right\rvert}

\theoremstyle{plain}
\newtheorem*{proposition}{Proposition}
\theoremstyle{remark}
\newtheorem*{remark}{Remark}

\title{\bfseries Palindromes on the $\tau$-circle\\[2pt]
\large A note for Palindrome Tau Day, \texttt{6/28/26}}
\author{Scott Duke Kominers%
\thanks{%
\textsc{Harvard Business School; Department of Economics and Center of
Mathematical Sciences and Applications, Harvard University; and a16z crypto.}
I used LLMs to assist in the preparation of this note, especially Claude~4.8 Opus. The subject matter and celebra$\tau$ion are my own; and of course any errors remain my responsibili$\tau$y.\smallskip%
\newline\indent%
\newline\indent%
\smallskip
}}
\date{}

\begin{document}
\maketitle
\vspace{-2.2em}
\begin{center}\small $62826$\end{center}
\vspace{0.6em}

\begin{abstract}
\noindent An integer palindrome is a self-reciprocal polynomial evaluated at its base,
so its roots are symmetric about the unit circle---where the coordinate is angle, in 
turns of $\tau$. Read this way, the date $\texttt{6/28/26}\to 62826$ secretly contains
the primitive cube roots of unity---at angle $\tau/3$---along with one further pair 
of roots on the circle.
\end{abstract}

\section{The date}

Tau Day is the 28th of June, because $\tau=2\pi=6.283\ldots$ and $6/28$ reads as $6.28$.
This year the date is also a palindrome: concatenating month, day, and year,
\[
\texttt{6/28/26}\;\longrightarrow\;6\,\Vert\,28\,\Vert\,26\;=\;62826,
\]
which reads the same backwards as forwards. The coincidence runs deeper. Scaled to five
digits, $\tau\cdot 10^{4}=62831.85\ldots$ sits between the palindromes $62826$ and
$62926$, and far nearer the first ($5.85$ against $94.15$). So $62826$---read with a
decimal point, as $6.2826$---is the five-digit palindrome closest to $\tau$.\footnote{One
might expect this to be automatic: surely the nearest palindrome is just the leading
digits of $\tau$, reflected? At five digits that recipe does give $62826$ (from $6,2,8$),
but it is not reliable---at four digits it returns $6226$, whereas $6336$ is in fact
closer to $\tau\cdot 10^{3}$. So the five-digit agreement (nearest palindrome, reflected 
digits, and the date, all at once) is a genuine coincidence, not a foregone conclusion.}

That much is numerology. The real link between palindromes and $\tau$ is geometric, and
it runs through the unit circle---which is what the rest of this note is about.

\section{From digits to the circle}

Write a base-$b$ palindrome with digits $(d_0,d_1,\dots,d_n)$, $d_k=d_{n-k}$, as the
value $P(b)$ of the polynomial
\[
P(x)=\sum_{k=0}^{n} d_k\,x^{k}.
\]
The symmetry of the coefficients makes $P$ \emph{self-reciprocal} (\emph{palindromic} as
a polynomial): $x^{n}P(1/x)=P(x)$. Equivalently, its roots are closed under inversion
$z\mapsto 1/z$---if $\rho$ is a root, then so is $1/\rho$.

The coefficients are real, so the roots are closed under conjugation too. Off the
unit circle they then come in quadruples $\{\rho,\bar\rho,1/\rho,1/\bar\rho\}$ (or pairs,
when $\rho$ is real), while a root lands \emph{on} the circle $\T=\{\abs{z}=1\}$ exactly
when its reciprocal and its conjugate coincide---that is, when $\abs{\rho}=1$. The circle
is the curve where reciprocal and conjugate agree.

That is where $\tau$ enters. A point of $\T$ is $e^{i\theta}$, described by its angle, 
as a fraction of a full turn, $\theta/\tau\in[0,1)$. The the cyclotomic polynomials 
$\Phi_n$ ($n\ge 2$) are the cleanest case: monic integer palindromes whose roots
sit exactly on $\T$, at rational angles $\tau\,k/n$. They are where ``integer palindrome''
and the constant $\tau$ coincide exactly. What remains is to see when a general
palindrome's roots reach the $\tau$-circle.

\section{62826}

We read the date as a polynomial over its base:
\[
62826 = P(10),\qquad P(x)=6x^{4}+2x^{3}+8x^{2}+2x+6 .
\]
To find $P$'s roots, we use the symmetry of the coefficients directly. Divide by $x^2$
(harmless, since $0$ is not a root) and collect the terms the palindrome symmetry pairs
together:
\[
\frac{P(x)}{x^{2}} = 6\bigl(x^{2}+x^{-2}\bigr)+2\bigl(x+x^{-1}\bigr)+8 .
\]
The right-hand side depends only on $u:=x+\tfrac1x$, since $x^{2}+x^{-2}=u^{2}-2$.
Substituting,
\[
\frac{P(x)}{x^{2}} = 6(u^{2}-2)+2u+8 = 6u^{2}+2u-4 = 2(3u-2)(u+1).
\]
The quartic has collapsed to a quadratic in $u$, with roots $u=-1$ and $u=\tfrac23$. Each
value of $u$ unfolds into a pair of $x$s through $x^{2}-ux+1=0$. The first, $u=-1$, gives
$x^{2}+x+1=\Phi_3(x)$, whose roots are the primitive cube roots of unity
$e^{\pm i\tau/3}$---a clean $\tau/3$ hiding inside our party-hat number. The second,
$u=\tfrac23$, gives $3x^{2}-2x+3$, whose roots $\dfrac{1\pm 2\sqrt2\,i}{3}$ also have
modulus $1$---$\bigl\lvert\tfrac{1\pm2\sqrt2\,i}{3}\bigr\rvert^{2}=\tfrac{1+8}{9}=1$---and
sit at angle $\arccos\tfrac13\approx0.196\,\tau\approx70.53^\circ$. The full factorization,
\[
\boxed{\,6x^{4}+2x^{3}+8x^{2}+2x+6 
\;=\; 2\,(x^{2}+x+1)\,(3x^{2}-2x+3)
\;=\;2\,\Phi_3(x)\,(3x^{2}-2x+3),\,}
\]
puts all four roots of $62826$ on the unit circle.

\begin{remark}
We call a palindrome \emph{unit} if all its roots lie on $\T$; $62826$ is an example. 
By Kronecker's theorem~\cite{Kronecker} a \emph{monic} integer polynomial with all roots
on $\T$ is a product of cyclotomics, with every angle a rational multiple of $\tau$---and 
$62826$ escapes that conclusion by not being monic. Its roots $\tfrac{1\pm2\sqrt2\,i}{3}$ are
not algebraic integers, and by Niven's theorem~\cite{Niven} their angle $\arccos\tfrac13$
is an \emph{irrational} multiple of $\tau$ (as $\cos\theta=\tfrac13\notin\{0,\pm\tfrac12,\pm1\}$).
So $62826$ sits on the edge between the cyclotomic palindromes and the rest: it has every root on
the circle, yet one pair at an angle no fraction of a turn will ever name.
\end{remark}

\section{Not just 62826}

Our calculation for $62826$ is an example of something general: the fold
$u=x+\tfrac1x$ halves the degree of any even palindrome, and yields a simple
test for which roots land on the circle.

\begin{proposition}
Let $P$ be a self-reciprocal polynomial of degree $2m$. Then a unique polynomial $Q$ of
degree $m$ satisfies
\[
P(x)=x^{m}\,Q\!\left(x+\tfrac1x\right),
\]
and each root $x$ of $P$ gives a root $u=x+\tfrac1x$ of $Q$. The root $x$ lies on the
unit circle if and only if $u$ is real and lies in $[-2,2]$; in that case $u=2\cos\theta$
for $x=e^{i\theta}$. (A self-reciprocal polynomial of odd degree has $-1$ as a
root; dividing by $x+1$ returns to the even case.)
\end{proposition}

\noindent The proposition is classical. The substitution $u=x+\tfrac1x$ is from the theory
of reciprocal equations, and $Q$ is the \emph{Chebyshev transform} of $P$ (the $s_j$ in
the proof below are Chebyshev polynomials in $u$). That $P$'s roots lie on the circle
exactly when $Q$'s lie in $[-2,2]$ is Lemma~1 of Lakatos~\cite{Lakatos}; see also
Marden~\cite{Marden}, the survey of Vieira~\cite{Vieira}, and Konvalina--Matache~\cite{KM}
on palindromes and the circle. The short argument:

\begin{proof}
Divide by $x^{m}$ and pair the indices $k=m\pm j$, using $d_{m+j}=d_{m-j}$:
\[
\frac{P(x)}{x^{m}}=\sum_{k=0}^{2m} d_k\,x^{k-m}
= d_m+\sum_{j=1}^{m} d_{m+j}\bigl(x^{j}+x^{-j}\bigr).
\]
Set $u=x+x^{-1}$. The functions $s_j:=x^{j}+x^{-j}$ obey $s_0=2$, $s_1=u$, and the
identity $x^{j+1}+x^{-(j+1)}=(x+x^{-1})(x^{j}+x^{-j})-(x^{j-1}+x^{-(j-1)})$ gives
$s_{j+1}=u\,s_j-s_{j-1}$. Hence each $s_j$ with $j\ge1$ is a monic polynomial in $u$ of
degree $j$, and
\[
Q(u):=d_m+\sum_{j=1}^{m} d_{m+j}\,s_j(u)
\]
is a polynomial of degree $m$ with leading coefficient $d_{2m}\neq0$, satisfying
$P(x)=x^{m}Q(x+\tfrac1x)$; uniqueness holds because $\{1,s_1,\dots,s_m\}$ is a basis for
the polynomials of degree $\le m$. For the second claim, if $x=e^{i\theta}$, then
$u=2\cos\theta\in[-2,2]$. Conversely, if $u\in(-2,2)$ is real, then $x^{2}-ux+1=0$ has
negative discriminant, while the product of its roots equals $1$ by Vieta's formulas, so
the two roots are complex conjugates of modulus $1$; the endpoints $u=\pm2$ give the double
roots $x=\pm1\in\T$; and any real $u$ with $\abs{u}>2$, or any non-real $u$, yields
$x\notin\T$. Thus $x\in\T\iff u\in[-2,2]$.
\end{proof}

Seen through the general proposition, our ``da$\tau$e'' reads in one line: $62826$ folds to
$6u^{2}+2u-4$, whose roots $-1$ and $\tfrac23$ both lie in the window $[-2,2]$; hence, all 
four roots of $62826$ sit on the $\tau$-circle. In general, a palindrome's roots land on 
the circle precisely to the extent that its folded half keeps its roots in $[-2,2]$.

\section{Coda}

A palindrome is a polynomial unchanged when its coefficients are read backwards---equivalently, 
when each root is sent to its reciprocal. With real coefficients that leaves
the roots symmetric about the unit circle, and the circle records angle in turns of $\tau$. 
So every palindrome is already a statement about $\tau$, and $62826$ is the one that says it on the
nose this June---right down to the perfect $\tau/3$ folded inside. For an even sharper rendition, 
wait eight hundred years---$6/28/2826$ reads as $6.282826$, the closest a palindrome date comes 
to $\tau$ (at least for the next many millennia).

\medskip
\noindent Happy Palindrome Tau Day, $\mathbb{QED}$!


\begin{thebibliography}{1}

\bibitem{KM}
J.~Konvalina and V.~Matache.
\newblock {Palindrome-polynomials with roots on the unit circle}.
\newblock {\em C.~R. Math. Acad. Sci. Soc. R. Can.}, 26(2):39--44, 2004.

\bibitem{Kronecker}
L.~Kronecker.
\newblock {Zwei S\"atze \"uber Gleichungen mit ganzzahligen Coefficienten}.
\newblock {\em J. Reine Angew. Math.}, 53:173--175, 1857.

\bibitem{Lakatos}
P.~Lakatos.
\newblock {On zeros of reciprocal polynomials}.
\newblock {\em Publ. Math. Debrecen}, 61(3-4):645--661, 2002.

\bibitem{Marden}
M.~Marden.
\newblock {\em {Geometry of Polynomials}}.
\newblock Mathematical Surveys, No.~3. American Mathematical Society,
  Providence, RI, 2nd edition, 1966.

\bibitem{Niven}
I.~Niven.
\newblock {\em {Irrational Numbers}}.
\newblock Carus Mathematical Monographs, No.~11. Mathematical Association of
  America, 1956.

\bibitem{Vieira}
R.~S. Vieira.
\newblock {Polynomials with Symmetric Zeros}.
\newblock In C.~S. Ryoo, editor, {\em {Polynomials -- Theory and Application}}.
  IntechOpen, 2019.
\newblock arXiv:1904.01940.

\end{thebibliography}
\end{document}